\documentclass[12pt]{amsart}
\usepackage{enumerate,amssymb,version,geometry,aliascnt}
\usepackage[all]{xy}
\usepackage{hyperref}

\hypersetup{
pdfauthor={Olivier Haution},
pdftitle={Degree formula for the Euler characteristic},
pdfkeywords={Grothendieck group, Euler characteristic, degree formula}
}

\DeclareMathOperator{\id}{id}
\DeclareMathOperator{\Spec}{Spec}
\DeclareMathOperator{\rank}{rank}
\DeclareMathOperator{\CH}{CH}

\DeclareMathOperator{\mult}{mult}
\DeclareMathOperator{\CK}{CK}
\DeclareMathOperator{\pgcd}{gcd}

\newcommand{\Tan}{T}
\newcommand{\Zz}{\mathbb{Z}}

\newcommand{\G}{K_0'}
\newcommand{\K}{K_0}

\newcommand{\Oc}{\mathcal{O}}
\newcommand{\Ec}{\mathcal{E}}

\newtheorem{theorem}{Theorem}[section]

\newaliascnt{proposition}{theorem}
\newtheorem{proposition}[proposition]{Proposition}
\aliascntresetthe{proposition}

\newaliascnt{lemma}{theorem}
\newtheorem{lemma}[lemma]{Lemma}
\aliascntresetthe{lemma}

\newaliascnt{corollary}{theorem}
\newtheorem{corollary}[corollary]{Corollary}
\aliascntresetthe{corollary}

\theoremstyle{definition}

\newaliascnt{remark}{theorem}

\aliascntresetthe{remark}

\newaliascnt{example}{theorem}

\aliascntresetthe{example}

\newaliascnt{definition}{theorem}
\newtheorem{definition}[definition]{Definition}
\aliascntresetthe{definition}

\begin{document}
\begin{abstract}
We give a proof of the degree formula for the Euler characteristic previously obtained by Kirill Zainoulline. The arguments used here are considerably simpler, and allow us to remove all restrictions on the characteristic of the base field.
\end{abstract}

\author{Olivier Haution}
\title{Degree formula for the Euler characteristic}
\email{olivier.haution at gmail.com}
\address{School of Mathematical Sciences\\
University of Nottingham\\
University Park\\
Nottingham\\
NG7 2RD\\ 
United Kingdom}
\date{\today}
\setcounter{tocdepth}{1}
\maketitle

\section*{Introduction}
The degree formula for the Euler characteristic says that if $f \colon Y \dashrightarrow X$ is a rational map, with $X$ and $Y$ projective connected smooth varieties of the same dimension $d$, then $X$ possesses a zero-cycle of degree 
\[
\tau_{d-1} \cdot(\chi(\Oc_Y)- \deg f \cdot \chi(\Oc_X)).
\]
Here $\deg f$ is the degree of the function fields extension (zero when $f$ is not dominant), $\chi$ the Euler characteristic, and $\tau_{d-1}$ the $d-1$-st Todd number (the denominator in the $d-1$-st Todd class, see \eqref{Todd}). This formula is useful to prove incompressibility properties of some varieties.\\

In the paper \cite{Zai-09}, where this formula is introduced, two distinct proofs are given, based on different results: 
\begin{enumerate}[(a)]
\item \label{gen} the generalized degree formula for algebraic cobordism,
\item \label{rost} or the Rost's degree formula.
\end{enumerate}
Both techniques are quite sophisticated, and moreover require to make some assumptions on the characteristic of the base field. It is known that \eqref{gen} (together with the universal property of algebraic cobordism) implies \eqref{rost}, but \eqref{rost} has the advantage of being known for some fields of positive characteristic, while \eqref{gen} requires to work over a field of characteristic zero.

Indeed \eqref{rost} is not proved at the moment when no information on the characteristic of the base field is available, even under the assumption of resolution of singularities. In \cite[Section~8]{reduced} we showed that the $p$-primary part of the result of \cite{Zai-09} can be obtained in arbitrary characteristic when one disposes of the so-called $p$-resolution of singularities. This suggests that this result, the degree formula for the Euler characteristic, does not lie at the same depth as the classical degree formula \eqref{rost}.\\

The purpose of this article is to give a simpler proof of this result, over any field. In contrast to \eqref{gen} or \eqref{rost}, the main ingredients here are the Grothendieck-Riemann-Roch theorem, and a small fraction of \cite{Mer-Ori}.\\

{\bf Acknowledgements.} The support of EPSRC Responsive Mode grant EP/G032556/1 is gratefully acknowledged. I thank Alexander Vishik for the useful discussions that we had on the subject of this paper.

\section{Notations}
\subsection{Varieties} We fix a base field $k$. A \emph{variety} is a finite type, separated, quasi-projective scheme over $k$. A morphism of varieties is a morphism of schemes over $k$. When $X$ is a smooth variety, we denote by $\Tan_X$ its tangent bundle.

\subsection{Grothendieck groups of schemes} Let $X$ be a noetherian scheme. We denote by $\G(X)$ (resp. $\K(X)$) the Grothendieck group (resp. ring) of coherent $\Oc_X$-modules (resp. locally-free $\Oc_X$-modules). 

If $f \colon Y \to X$ is a flat morphism of noetherian schemes, then it induces a pull-back $f^* \colon \G(X) \to \G(Y)$. 

If $f \colon Y \to X$ is a projective morphism of varieties, then it induces a push-forward $f_* \colon \G(Y) \to \G(X)$.

\subsection{Poincare homomorphism}
There is a natural map 
\[
\K(X) \to \G(X)\; \colon \; x \mapsto x \cap [\Oc_X]
\]
which is an isomorphism when $X$ is regular (i.e. for every point $x$ of $X$ the local ring $\Oc_{X,x}$ is regular).

\subsection{Rank homomorphism} When $X$ is connected, there is a ring homomorphism, sending a vector bundle to its rank
\[
\rank \colon \K(X) \to \Zz.
\]

\subsection{First Chern class} Let $E$ be a vector bundle over a connected, noetherian scheme $X$. We denote its first Chern class by
\[
c_1(E)=\rank E- [E^\vee] \in \K(X).
\]

\subsection{Subgroup of generically trivial classes} Let $X$ be an integral variety. We denote by $\G(X)^{(1)}$ the subgroup of $\G(X)$ generated by the elements $i_*[\Oc_W]$ with $i \colon W \hookrightarrow X$ a non-dominant closed embedding of varieties. We have an exact sequence
\begin{equation}
\label{eq:exact}
0 \to \G(X)^{(1)} \to \G(X) \xrightarrow{\eta^*} \G(\Spec(k(X))) \to 0,
\end{equation}
where $\eta \colon \Spec(k(X)) \to X$ is the generic point.

\subsection{Euler characteristic} Let $X$ be a projective variety, and $x \colon X \to \Spec(k)$ its structural morphism. The \emph{Euler characteristic} of a coherent $\Oc_X$-module $\Ec$ is
\[
\chi(\Ec)=\sum_{i\geq 0} (-1)^i \dim_k H^i(X,\Ec)=x_*[\Ec].
\]
For the last equality we have used the identification $\G(\Spec(k))=\Zz$.\\

\section{Degree formula for $K$-theory}
\label{sect:deg}
\begin{lemma}
\label{lemm:fisrtc}
Let $V$ be a vector bundle over a smooth connected variety $X$. Then there exists smooth varieties $Z_i$ of dimension $\dim X-1$, projective morphisms $f_i \colon Z_i \to X$, and integers $n_i$, such that we have in $\G(X)$
\[
c_1(V)\cap[\Oc_X]=\sum_i n_i \cdot (f_i)_*[\Oc_{Z_i}].
\]
\end{lemma}
\begin{proof}
Consider the theory $K=\K(-)[t,t^{-1}]$ of \cite[Example 2.3]{Mer-Ori}. When $f \colon Y \to X$ is a projective morphism of smooth connected varieties, then the push-forward $f^K_*\colon K(Y) \to K(X)$ along $f$ is given by the formula, for $x \in \K(Y)$,
\begin{equation}
\label{eq:e}
f^K_*(x\cdot t^k)=f_*(x)\cdot  t^{k+\dim Y-\dim X}.
\end{equation}
Here $f_*(x)\in \K(X)$ is the element defined by $f_*(x) \cap [\Oc_X]=f_*(x\cap [\Oc_Y]) \in \G(X)$.

We have in $K(X)$ (see \cite[Example 3.1]{Mer-Ori})
\begin{equation}
\label{d}
c_1^K(V)=(\rank V  - [V^\vee])\cdot t^{-1}.
\end{equation}
On the other hand, we have, by \cite[Theorem 9.8]{Mer-Ori},
\begin{equation}
\label{c}
c_1^K(V)=\sum_i n_i \cdot (f_i)^K_*(1_{Z_i}).
\end{equation}
with $n_i, f_i, Z_i$ as requested. Note that this element belongs to $K^1(X)=\K(X)\cdot t^{-1}$, therefore in view of \eqref{eq:e} we can choose the varieties $Z_i$ so that $\dim Z_i=\dim X -1$.

We obtain the result by applying to \eqref{d} and \eqref{c} the composite
\[
K(X)=\K(X)[t,t^{-1}] \to \K(X)\cdot t^{-1}=\K(X) \to \G(X). \qedhere
\]
\end{proof}

\begin{proposition}
\label{prop:main}
Let $X$ be a smooth connected variety. Then  $\G(X)^{(1)}$ is additively generated by elements of type $f_*[\Oc_Z]$ with $Z$ a smooth variety such that $\dim Z = \dim X-1$, and $f \colon Z \to X$ a projective morphism.
\end{proposition}
\begin{proof}
Since $X$ is a regular variety, the map $\K(X) \to \G(X)$ is an isomorphism. Under this identification, the subgroup $\G(X)^{(1)}$ corresponds to the kernel of the rank homomorphism. Any element $x$ of this kernel can be written as $[E] - [F]$, for some vector bundles $E$ and $F$ on $X$, having the same rank $n$. Then
\[
x=(n-[F])-(n-[E])=c_1(F^\vee) - c_1(E^\vee). 
\]
Finally we apply \autoref{lemm:fisrtc} above.
\end{proof}

\begin{definition}
\label{def:deg}
Let $f \colon Y \to X$ be a projective morphism of varieties, with $X$ integral. Consider the generic fiber $Y \times_X k(X)$ as a variety over $k(X)$, and define an integer 
\[
\deg{f}=\chi(\Oc_{Y \times_X k(X)}),
\]
being understood that $\chi(\emptyset)=0$. 
\end{definition}

\begin{lemma}
\label{lemm:deg}
Let $f \colon Y \to X$ be a projective morphism, with $X$ integral. Then
\[
f_*[\Oc_Y]-\deg f \cdot [\Oc_X] \in \G(X)^{(1)}.
\]
\end{lemma}
\begin{proof}
This follows from the sequence \eqref{eq:exact}, and from the commutative diagram
\[ \xymatrix{
\G(Y)\ar[r] \ar[d]_{f_*} & \G(\Spec(Y_{k(X)}))  \ar[rd]^{\chi}& \\ 
\G(X) \ar[r] & \G(\Spec(k(X))) \ar@{=}[r]& \Zz 
}\]
\end{proof}

\begin{theorem}
\label{th:main}
Let $f \colon Y \to X$ be a projective morphism, with $X$ a smooth connected variety. Then we have in $\G(X)$
\[
f_*[\Oc_Y]=\deg f \cdot [\Oc_X] + \sum_i n_i \cdot (f_i)_*[\Oc_{Z_i}],
\]
for some smooth varieties $Z_i$ of dimension $\dim X -1$, projective morphisms $f_i \colon Z_i \to X$, and integers $n_i$.
\end{theorem}
\begin{proof}
This results from the combination of \autoref{lemm:deg} and \autoref{prop:main}. 
\end{proof}

\section{The Euler characteristic}
\label{sect:euler}
Given a rational number $\alpha$, let us denote by $[\alpha]$ the greatest integer $\leq \alpha$. If $d \geq 0$ is an integer, the $d$-th \emph{Todd number} is
\begin{equation}
\label{Todd}
\tau_d=\prod_{p \text{ prime}} p^{[d/(p-1)]}.
\end{equation}
These numbers appear as denominators in the Todd class. We have $\tau_d | \tau_e$ whenever $d\leq e$.\\

The next lemma is an immediate consequence of the Grothendieck-Riemann-Roch theorem.
\begin{lemma}[{\cite[Lemma~3.6]{Zai-09}}]
\label{lemm:chi}
If $X$ is a smooth projective variety, then it possesses a zero-cycle of degree $\tau_{\dim X} \cdot \chi(\Oc_X)$.
\end{lemma}

If $X$ is a projective variety, we denote by $n_X$ the positive integer such that 
\[
\deg \CH(X)=n_X \cdot \Zz.
\]
This integer coincides with the greatest common divisor of the degrees of closed points of $X$.

\begin{proposition}
\label{prop:deg}
Let $f \colon Y \to X$ be a projective morphism. Assume that $X$ is smooth, projective and connected. Then we have, using \autoref{def:deg},
\[
\tau_{\dim X-1} \cdot \chi(\Oc_Y)=\deg f \cdot \tau_{\dim X-1} \cdot \chi(\Oc_X) \mod n_X.
\]
\end{proposition}
\begin{proof}
Project the formula of \autoref{th:main} to $\G(\Spec(k))=\Zz$. This gives
\[
\chi(\Oc_Y)=\deg f \cdot \chi(\Oc_X) + \sum_i n_i \cdot \chi(\Oc_{Z_i}).
\]
Note that $n_X |n_{Z_i}$. Since every $Z_i$ is smooth and of dimension $<\dim X$, \autoref{lemm:chi} gives the result.
\end{proof}

Let $X,Y$ be projective integral varieties. A \emph{correspondence} $Y \leadsto X$ is an element $\gamma \in \CH(Y \times X)$. The \emph{multiplicity} $\mult \gamma$ is the image of $\gamma$ under the map
\[
\CH(Y \times X) \to \CH(k(Y) \times_k X) \to \CH(k(Y))=\Zz.
\]
The \emph{transpose} ${}^t \gamma$ of $\gamma$ is the correspondence $X \leadsto Y$ corresponding to the image of $\gamma$ under the morphism exchanging factors.

When $\gamma=[\Gamma]$ for some integral closed subvariety $\Gamma$ of $Y \times X$, then $\mult \gamma$ can be non-zero only if $\dim \Gamma=\dim Y$. In this case, $\mult \gamma$ coincides with the integer $\deg (\Gamma \to Y)$ of \autoref{def:deg}.

\begin{lemma}
\label{lemm:corr}
Let $\gamma \colon Y \leadsto  X$ be a correspondence between projective integral varieties, with $Y$ smooth. Then $n_X | \mult \gamma \cdot n_Y$.
\end{lemma}
\begin{proof}
The map $\mult \colon \CH(Y \times X)\to \Zz$ is linear and vanishes on cycles of dimension $\neq \dim Y$, hence we can assume that $\gamma=i_*[\Gamma]$, where $i \colon \Gamma \hookrightarrow Y \times X$ is an integral closed subvariety of dimension $\dim Y$. We have a diagram
\[ \xymatrix{
\CH(Y \times \Gamma) \ar[rr]^{(\id_Y \times i)_*} && \CH(Y \times Y \times X) \ar[d]^{(\id_{Y\times Y} \times x)_*} \ar[rr]^{(\Delta\times \id_X)^*}&& \CH(Y\times X) \ar[r]^{(y \times \id_X)_*}\ar[d]^{(\id_Y \times x)_*}& \CH(X)\ar[d]^{x_*} \\ 
\CH(Y) \ar[u]_{(\id_Y \times \gamma)^*} \ar[rr]_{\mult \gamma \cdot(\id_Y \times y)^*}&& \CH(Y \times Y)  \ar[rr]_{\Delta^*}&& \CH(Y) \ar[r]_{y_*}& \CH(\Spec(k))
}\]
where $\Delta \colon Y \hookrightarrow Y \times Y$ is the diagonal embedding, and $x \colon X \to \Spec(k)$, $y \colon Y \to \Spec(k)$, $\gamma \colon \Gamma \to \Spec(k)$ the structural morphisms. The square on the left is commutative because the push-forward along the projective morphism $\Gamma \to Y$ sends $[\Gamma]$ to $\mult \gamma\cdot [Y]$. Commutativity of the other squares is clear. We see that the bottom composite is $\mult \gamma\cdot y_*$ and factors through $x_*$.
\end{proof}

\begin{theorem}
\label{th:chi}
Let $\gamma\colon Y \leadsto X$ be a correspondence between smooth, projective, connected varieties of the same dimension $d$. Then we have
\[
\mult \gamma \cdot \tau_{d-1} \cdot \chi(\Oc_Y)=\mult {}^t \gamma \cdot \tau_{d-1} \cdot \chi(\Oc_X) \mod \pgcd(n_Y,n_X).
\]
\end{theorem}
\begin{proof}
We can assume as above that $\gamma$ is represented by an integral $d$-dimensional closed subvariety $\Gamma$ of $Y\times X$. The result then follows from the application of \autoref{prop:deg} to the projective morphisms $\Gamma \to X$ and $\Gamma \to Y$.
\end{proof}

Let $f\colon Y \dashrightarrow X$ be a rational map of integral projective varieties. The closure of its graph defines a correspondence $\gamma_f$ of multiplicity one. We define the integer $\deg f$ as $\mult {}^t \gamma_f$. This is compatible with \autoref{def:deg}. 

Combining \autoref{th:chi} and \autoref{lemm:corr}, we obtain the following generalization of \cite[Theorem~5.7]{Zai-09}.

\begin{corollary}
\label{cor:rat}
Let $f \colon Y  \dashrightarrow X$ be a rational map of projective, smooth, connected varieties of the same dimension $d$. Then we have
\[
\tau_{d-1} \cdot \chi(\Oc_Y)=\deg f \cdot \tau_{d-1} \cdot \chi(\Oc_X) \mod n_X.
\]
\end{corollary}

\section{Remarks and consequences}
\subsection{Incompressibility}
Using \autoref{cor:rat}, one can generalize the statements of \cite[6.2,\ldots,6.6]{Zai-09} to base fields of arbitrary characteristic.

\subsection{Resolution of singularities}
When resolution of singularities is available, one can obtain \autoref{prop:main} or \autoref{lemm:chi}, and therefore \autoref{prop:deg} for singular $X$. One can thus remove the smoothness assumptions in \autoref{th:chi}.

When the dimension of $X$ is $<p(p-1)$, where $p$ is the characteristic of the base field, one can also use \cite{firstst} to obtain the same result, over any field.

\subsection{Perfect fields of positive characteristic}
A statement of Rost (\cite[Corollary~1]{Ros-On-08}) says that for any projective variety $X$ over a perfect field of positive characteristic $p$, one has
\[
v_p(n_X) \leq v_p(\chi(\Oc_X)).
\]

It follows that the $p$-primary content of \autoref{th:chi} is empty when the base field is perfect.

\subsection{Generalized degree formula for periodic multiplicative theories}
By \cite[Theorem~4.2.10]{LM-Al-07}, we know that $K$-theory is the universal oriented weak cohomology theory with periodic multiplicative formal group law. Therefore \autoref{th:main} implies
\begin{proposition}
Let $A$ be a oriented weak cohomology theory (\cite[Definition~4.1.13]{LM-Al-07}) over an arbitrary field $k$. Assume that the formal group law of $A$ is $F(x,y)=x+y-\alpha \cdot xy$, for some invertible element $\alpha \in A(\Spec(k))$. Then for any projective morphism of smooth connected varieties $f\colon Y \to X$, we have in $A(X)$
\[
\alpha^{-\dim Y} \cdot f_*(1_Y)=\deg f \cdot \alpha^{-\dim X} \cdot 1_X + \sum_i n_i \cdot \alpha^{1-\dim X} \cdot (f_i)_*(1_{Z_i}).
\]
for some smooth varieties $Z_i$ of dimension $\dim X-1$, projective morphisms $f_i \colon Z_i \to X$, and integers $n_i$. 
\end{proposition}

In particular we obtain that the subgroup of $A(X)$ generated by projective push-forwards of fundamental classes of smooth varieties of arbitrary dimensions (``the image of cobordism'') coincides with the $\Zz[\alpha]$-submodule of $A(X)$ generated by $1_X$ and the projective push-forwards of fundamental classes of smooth varieties of dimensions $<\dim X$.

\subsection{Generalized degree formula for connective $K$-theory}
Connective $K$-theory $\CK_{p,q}$ has been introduced, over any field, in \cite{Cai}. It appears that the degree formula for this theory is equivalent to the degree formula for $K$-theory.

In order to make a precise statement, we use the notations of \cite{Cai}. In addition, for an integral variety $X$, we denote by $[X]$ the element $[\Oc_X]$ considered as an element of $\CK_{\dim X,-\dim X}(X)$. Thus the Bott element $\beta \in \CK_{1,-1}(\Spec(k))$ is $p_*[{\mathbb{P}^1}]$, where $p\colon \mathbb{P}^1\to \Spec(k)$ is the projective line.

\begin{proposition}
Let $f\colon Y \to X$ be a projective morphism of integral varieties, with $X$ smooth. Assume that $c=\dim X -\dim Y$ is $\leq 0$. Then we have in $\CK_{\dim Y,-\dim Y}(X)$
\[
f_*[Y]=\deg f \cdot \beta^{-c} \cdot [X] + \sum_i n_i \cdot \beta^{1-c} \cdot (f_i)_*[Z_i].
\]
for some smooth varieties $Z_i$ of dimension $\dim X-1$, projective morphisms $f_i \colon Z_i \to X$, and integers $n_i$. 
\end{proposition}
\begin{proof}
It follows from the construction of $\CK$-groups, and from the inequality $\dim Y \geq \dim X$, that the natural map $\CK_{\dim Y,-\dim Y}(X) \to \G(X)$ is an isomorphism. This maps sends the formula of the proposition to the formula of \autoref{th:main}.
\end{proof}

\bibliographystyle{alpha}

\end{document}